\documentclass[a4paper,12pt, reqno]{amsart}
\usepackage{amsmath}
\usepackage{amssymb}
\usepackage{amsxtra}
\usepackage{amsthm}
\usepackage{mathrsfs}
\usepackage{mathtools}
\usepackage{graphicx}
\usepackage{epstopdf}
\usepackage{subfigure}
\usepackage{float}
\usepackage{xcolor,graphicx,float}
\usepackage{enumerate}
\def\EquationsBySection{\def\theequation
	{\thesection.\arabic{equation}}%
	\@addtoreset{equation}{section}}

\usepackage{indentfirst}
\newtheorem{remark}{Remark}[section]

\newtheorem{theorem}[remark]{Theorem}

\newcommand\old[1]{}
\makeatother
\EquationsBySection
\usepackage{xcolor}
\usepackage[colorlinks,citecolor=blue,pagebackref,hypertexnames=false]{hyperref}

\title[Lane-Emden systems involving the Logarithmic Laplacian] 
      {Symmetry of positive solutions for Lane-Emden systems involving the Logarithmic Laplacian}
      \author{Rong Zhang, Vishvesh Kumar and Michael Ruzhansky}
\address[ Rong Zhang]{School of Mathematical Sciences,
Nanjing Normal University, 
   Nanjing, 210023, China \newline and \newline Department of Mathematics: Analysis, Logic and Discrete Mathematics, Ghent University, Ghent, Belgium}
\email{zhangrong@nnu.edu.cn / rongzhangnnu@163.com}
\address[Vishvesh Kumar]{Department of Mathematics: Analysis, Logic and Discrete Mathematics, Ghent University, Ghent, Belgium}
\email{vishveshmishra@gmail.com / vishvesh.kumar@ugent.be}

\address[Michael Ruzhansky]{Department of Mathematics: Analysis, Logic and Discrete Mathematics, Ghent University, Ghent, Belgium\newline and \newline
School of Mathematical Sciences, Queen Mary University of London, United Kingdom}
\email{michael.ruzhansky@ugent.be}

\keywords{ Fractional Laplacian; Logarithmic Laplacian; Logarithmic Symbol; Symmetry and monotonicity;
The direct method of moving planes; Lane-Emden system.}
\subjclass{ 35R11, 35B06, 35B50, 35B51, 35D30.}

\begin{document}
\maketitle
\allowdisplaybreaks
\begin{abstract}
    We study the Lane-Emden system involving
the logarithmic Laplacian:
$$
\begin{cases}
\ \mathcal{L}_{\Delta}u(x)=v^{p}(x) ,& x\in\mathbb{R}^{n},\\
\ \mathcal{L}_{\Delta}v(x)=u^{q}(x) ,& x\in\mathbb{R}^{n},
\end{cases}
$$
where $p,q>1, n\geq 2$ and $\mathcal{L}_{\Delta}$ denotes the logarithmic Laplacian arising as a formal derivative $\partial_s|_{s=0}(-\Delta)^s$ of fractional Laplacians at $s=0.$ By using a direct method of moving planes for the logarithmic Laplacian, we obtain the symmetry and monotonicity of the positive solutions to the Lane-Emden system.
We also establish some key ingredients needed in order to apply the method of moving planes such as the maximum principle for anti-symmetric functions, the narrow region principle, and decay at infinity. Further,  we discuss such results for a generalized system of the Lane-Emden type involving the logarithmic Laplacian.   

\end{abstract}


\section{Introduction}

Triggered by several significant applications to model physical and natual phenomena and their closed connections with probablity theory and  stochastic processes, there has been a growing interest to study nonlinear nonlocal operators and many analytical and qualitative properties of solutions to the partial differential equations involving those nonlocal operators (see \cite{t1,t2,t3,t4,t5,t6,t7}). One of the most prominent candidates in the list of nonlocal operators of positive differential order is the fractional 
Laplacian on $\mathbb{R}^n$ \cite{t3,chen4,t7,t6,x1,17}. It is certainly near to impossible to mention all the works but we refer to \cite{chen4,cpam,lu,x1,17,r1,r2} and reference therein for many seminal studies in this area of research. 

We recall that, the fractional Laplacian $(-\Delta)^{s}, \,s\in (0, 1)$, on $\mathbb{R}^{n}$ is a nonlocal pseudo-differential operator defined on $u \in C_c^2(\mathbb{R}^n)$ by
\begin{equation}\label{q}
\begin{aligned}
(-\Delta)^{s}u(x)&=C_{n,s}P.V.\int_{\mathbb{R}^{n}}\frac{u(x)-u(y)}{|x-y|^{n+2s}}dy =C_{n,s}\lim_{\epsilon\rightarrow0}\int_{\mathbb{R}^{n}\backslash B_{\epsilon}(x)}\frac{u(x)-u(y)}{|x-y|^{n+2s}}dy,
\end{aligned}
\end{equation}
where  $P.V.$ stands for the Cauchy principal value, $B_{\epsilon}(x)$ is the open ball with centre $x \in \mathbb{R}^{n}$ and radius
$\epsilon>0$, and $$C_{n,s}=2^{2s}\pi^{-\frac{n}{2}\frac{\Gamma(\frac{n+2s}{2})}{-\Gamma(-s)}}=2^{2s}\pi^{-\frac{n}{2}\frac{\Gamma(\frac{n+2s}{2})}{\Gamma(1-s)}}$$ 
is a normalized positive constant depending on $n$ and $s$ (see \cite{17}) so that we have the following definition of  the operator $(-\Delta)^{s}$ via the Euclidean Fourier transform:
$$\widehat{(-\Delta)^{s}u(x)}=|\xi|^{2s}\hat{u}(\xi),\ u\in S,$$
where $\hat{u}$ is the Euclidean Fourier transform of $u$.


The non-locality of the fractional Laplacian makes it difficult to investigate. In order to overcome this difficulty, the extension method was firstly introduced by Caffarelli and Silvestre in \cite{a6}. The main idea of the extension method is to reduce the nonlocal problem to a local one in higher dimensions. Namely, for a function $u:\mathbb{R}^{n}\rightarrow \mathbb{R}$, we consider the extension $U:\mathbb{R}^{n}\times[0,\infty)\rightarrow \mathbb{R}$ that satisfies
$$
\begin{cases}
\ \text{div}(y^{1-2s}\nabla U)=0,& (x,y)\in \mathbb{R}^{n}\times[0,\infty),\\
\ U(x,0)=u(x), & x\in \mathbb{R}^{n}.
\end{cases}
$$
Then we have
$$(-\Delta)^{s}u(x)=-C_{n,s}\lim_{y\rightarrow0^{+}}y^{1-2s}\frac{\partial U}{\partial y},\ x\in\mathbb{R}^{n}.$$

The method of moving planes has
become a powerful tool in studying qualitative properties such as symmetry and monotonocity of solutions to nonlinear elliptic
equations and systems involving nonlocal operators. The method of moving planes, that goes back to Alexandrov and J. Serrin \cite{Serrin}, has been developed by many researchers including the seminal work of Chen \cite{chen91}. We refer to \cite{ GNN, zhuo,jems,ampa,cpa,chen2,chen3,am, Guo, Fra2, Fra1} and references therein for an overview. Of our particular interest is  the following Lane-Emden system in $\mathbb{R}^{n}$:
\begin{equation}\label{p1}
\begin{cases}
\ (-\Delta)^{s} u(x)=v^{p}(x),\\
\ (-\Delta)^{s} v(x)=u^{q}(x).
\end{cases}
\end{equation}

When $s=1$, the well-known Lane-Emden conjecture states that, in the subcritical case (i.e. $\frac{1}{p+1}+\frac{1}{q+1}>\frac{n-2}{n}$), system (\ref{p1}) does not have positive classical solutions. The conjecture is known to be true for radial solutions in all dimensions (see \cite{30,33}).
For non-radial solutions, and $n\leq2$, the conjecture is a consequence of known results (see \cite{31,33}).
It is known that if $(p,q)$ in critical or supercritical case (i.e. $\frac{1}{p+1}+\frac{1}{q+1}\leq\frac{n-2}{n}$), then system (\ref{p1}) admits some positive classical solutions on $\mathbb{R}^{n}$ (see \cite{40}).

When $0<s<1$, by using the direct method of moving planes, Cheng et al. \cite{cheng} obtained that in the subcritical case (i.e. $1<p,q<\frac{n+2s}{n-2s}$), system \eqref{p1} does not have any positive solutions, in the critical case (i.e. $p=q=\frac{n+2s}{n-2s}$), the solutions are radial symmetric about some point in $\mathbb{R}^n$.

In the special case $p=q$, $u=v$, system \eqref{p1} will become
\begin{equation}\label{p2}
(-\Delta)^{s} u(x)=u^{p}(x),\quad x\in \mathbb{R}^{n}.
\end{equation}

Chen, Li and Li \cite{chen1} developed a direct method of moving planes for the fractional Laplacian. By using the method of moving planes, they established symmetry and monotonicity for positive solutions of \eqref{p2}. Namely, in the critical case $p=\frac{n+2s}{n-2s}$, the solution is radially symmetric and monotone decreasing about some point, and in the subcritical case $1<p<\frac{n+2s}{n-2s}$, the equation \eqref{p2} does not have any positive solutions.

The main aim of this paper to begin an investigation about the qualitative properties of solutions to the nonlinear equations and systems involving the logarithmic Laplacian on $\mathbb{R}^n,$ defined  as a formal derivative $\partial_s|_{s=0}(-\Delta)^s$ of fractional Laplacians at $s=0.$ 
 The logarithmic Laplacian $\mathcal{L}_{\Delta}$ was introduced by Chen and Weth
\cite{chen}. On compactly supported Dini continuous functions, the logarithmic Laplacian $\mathcal{L}_{\Delta}$ is defined by the following pointwise evaluation
\begin{equation}\label{a1}
\begin{aligned}
\mathcal{L}_{\Delta}u(x)&=\frac{d}{ds}|_{s=0}[(-\Delta)^{s}u](x)\\
&=C_{n}P.V.\int_{\mathbb{R}^n}\frac{u(x)1_{B_{1}(x)}(y)-u(y)}{|x-y|^{n}}dy+\rho_{n}u(x),
\end{aligned}
\end{equation}
where
$$C_{n}:=\pi^{-\frac{n}{2}}\Gamma(\frac{n}{2}),\ \rho_{n}:=2\log2+\psi(\frac{n}{2})-\gamma.$$
Here $\Gamma$ is the Gamma function, $\psi=\frac{\Gamma'}{\Gamma}$ is the Digamma function, $\gamma=-\Gamma'(1)$ is the Euler Mascheroni constant,  $B_{1}(x)$ is an open ball with $x$ as its center and $1$ as its radius and $1_{B_{1}(x)}$ is the characteristic function of $B_{1}(x)$ in $\mathbb{R}^n$.  We note that (see \cite[Remark 1.2]{chen}) $\rho_n=-2\gamma+\sum\limits_{k=1}^{(n-1)/2} \frac{2}{2k-1}$ for $n$ odd and $\rho_n=2(\log 2-\gamma)+\sum\limits_{k=1}^{(n-2)/2} \frac{1}{k}$ for $n$ even. It is easy to see that $\rho_n \geq 0$ for every $n\geq 2.$

 Throughout the paper, we will assume that $n \geq 2$ so that $\rho_n \geq 0.$ If there is no confusion, we will not write P.V. in the definition of $\mathcal{L}_{\Delta}$ from this point.

Using the expression (\ref{a1}), one can define $\mathcal{L}_{\Delta}$ for a quite  large class of functions $u$. To illustrate this, we define, for $s\in\mathbb{R}$, the space $L_{s}^{1}(\mathbb{R}^{n})$ as the space of locally integrable functions $u:\mathbb{R}^{n}\rightarrow\mathbb{R}$ such that
$$\|u\|_{L_{s}^{1}(\mathbb{R}^{n})}:=\int_{\mathbb{R}^{n}}\frac{|u(x)|}{(1+|x|)^{n+2s}}dx<+\infty.$$
It was proved by \cite[Proposition 1.3]{chen} that for $u\in L_{0}^{1}(\mathbb{R}^{n})$ which is also Dini continuous at some $x\in \mathbb{R}^{n}$, the quantity $[\mathcal{L}_{\Delta}u](x)$ is well defined by the formula (\ref{a1}). Furthermore, as it was shown in \cite{chen}, $\mathcal{L}_{\Delta}$ is a weakly singular operator having an intrinsic scaling property of the logarithmic order with the Fourier symbol $2\log|\xi|.$ They also establish that if  $u\in C_{c}^{\vartheta}(\mathbb{R}^{n})$ for some $\vartheta>0$, then
\begin{equation}\label{a2}
\mathcal{L}_{\Delta}u(x)=\lim_{s\rightarrow0^{+}}\frac{d}{ds}(-\Delta)^{s}u.
\end{equation}

The strong maximum principle for pointwise solution on bounded Lipschitz domain  of $\mathbb{R}^n$ was established by Chen and Weth \cite[Proposition 4.1]{chen}. 
Among other things, they (\cite{chen}) also characterized the asymptotics of principal Dirichlet eigenvalues and eigenfunctions of $(-\Delta)^{s}$ as $s\rightarrow0$ and proved a Faber-Krahn type inequality. The spectral properties of $\mathcal{L}_{\Delta}$ were investigated by Laptev and Weth \cite{LW}. The regularity properties of eigenfunctions of $\mathcal{L}_{\Delta}$ were established in \cite{FJW}. In \cite{HVS}, the authors showed that least-energy solutions of the fractional Dirichlet problem converge to a nontrivial nonnegative
least-energy solution of a limiting problem in terms of the logarithmic Laplacian by using variational methods, uniform energy-derived estimates,
and   a new logarithmic-type Sobolev inequality.
 Very recently, a direct method of moving plane for the logarithmic Laplacian on $\mathbb{R}^n$ was developed by Zhang and Nie \cite{aml} (see \cite{Liu18} for bounded domains). Using this they obtained symmetry and monotonocity of the positive solution to a nonlinear equation involving the logarithmic Laplacian.



In this paper, we consider the following Lane-Emden system involving Logarithmic Laplacian on $\mathbb{R}^n$:
\begin{equation}\label{c1}
\begin{cases}
\ \mathcal{L}_{\Delta}u(x)=v^{p}(x) ,& x\in\mathbb{R}^{n},\\
\ \mathcal{L}_{\Delta}v(x)=u^{q}(x) ,& x\in\mathbb{R}^{n},
\end{cases}
\end{equation}
where $p,q>1$ and $\mathcal{L}_{\Delta}$ is given by \eqref{a1}.    By using a direct method of moving planes for the Logarithmic Laplacian, we obtain the symmetry and monotonicity of the positive solutions to the system \eqref{c1}. The following theorem is the main result of this paper.

\begin{theorem}\label{th1}
Let $u,v\in L_{0}^{1}(\mathbb{R}^{n})$ be a solution of system \eqref{c1} such that
\begin{equation}\label{b2}
u(x)=O\bigg(\frac{1}{|x|^{\alpha}}\bigg),\ 
v(x)=O \bigg(\frac{1}{|x|^{\beta}}\bigg),\quad|x|\rightarrow\infty,
\end{equation}
for some $\alpha,\beta>0$. Then, $u$ and $v$ are radially symmetric and monotone decreasing about some point $x^{o}$ in $\mathbb{R}^{n}$.
\end{theorem}

Now, we consider the generalized Lane-Emden elliptic system including $m$ equations which generalizes the system \eqref{c1} as follows:
\begin{equation}\label{m}
\begin{cases}
\ \mathcal{L}_{\Delta}u_{1}(x)=u_{2}^{p_{1}}(x) ,& x\in\mathbb{R}^{n},\\
\ \mathcal{L}_{\Delta}u_{2}(x)=u_{3}^{p_{2}}(x) ,& x\in\mathbb{R}^{n},\\
\ \qquad\qquad \cdot\\
\ \qquad\qquad\cdot\\
\ \qquad\qquad\cdot\\
\ \mathcal{L}_{\Delta}u_{m}(x)=u_{1}^{p_{m}}(x) ,& x\in\mathbb{R}^{n},
\end{cases}
\end{equation}
where $p_{i}>1, i=1,\cdot\cdot\cdot,m$. 
We have the following result concerning the radial symmetry and monotonocity of solutions to the system \eqref{m}.

\begin{theorem}\label{thm}
Let $u_{i}\in L_{0}^{1}(\mathbb{R}^{n}), i=1,\cdot\cdot\cdot,m$ be a solution of system \eqref{m} such that
\begin{equation}\label{n}
u_{i}(x)=O\bigg(\frac{1}{|x|^{\alpha_{i}}}\bigg),\ for\ some \ \alpha_{i}>0, \quad i=1,\cdot\cdot\cdot,m .
\end{equation}
Then, $u_{i}$ are radially symmetric and monotone decreasing about some point $x^{o}$ in $\mathbb{R}^{n}$.
\end{theorem}

The paper is organised as follows: In Section 3, we present 
some key ingredients needed in the process of applying the methods of moving planes such as the maximum principle for anti-symmetric functions, narrow region principle and decay at infinity. In Section 4, by using a direct method of moving planes, we establish  the proofs of Theorem
\ref{th1} and Theorem \ref{thm}.

\section{Key ingredients for the method of moving planes}
\label{keying}
This section is devoted to developing basic and key results needed to apply the method of moving planes for establishing the proof of our main result in the next section. 

We first present some basic notation and nomenclatures which will be beneficial for the rest of the paper. 

Select any direction as the $x_{1}$-direction, consider the $x_{1}$-direction, for $x=(x_{1},\cdot\cdot\cdot,x_{n})=(x_{1},x')\in\mathbb{R}^{n}$. Denote the moving planes
$$T_{\lambda}=\{x\in\mathbb{R}^{n}\mid x_{1}=\lambda,\ \lambda\in\mathbb{R}\}$$
and the region to the left by
$$\Sigma_{\lambda}=\{x\in\mathbb{R}^{n}\mid x_{1}<\lambda\}.$$ 
Let
$$x^{\lambda}=(2\lambda-x_{1},x')$$
is the reflection of the point $x=(x_{1},x')$ about the plane $T_{\lambda}$.

Define
$$u_{\lambda}(x):=u(x^{\lambda}),\quad v_{\lambda}(x):=v(x^{\lambda}),$$
and denote
$$U_{\lambda}(x)=u_{\lambda}(x)-u(x),\quad V_{\lambda}(x)=v_{\lambda}(x)-v(x).$$
It is easy to see that $U_\lambda$ and $V_\lambda$ are anti-symmetric functions, that is, $U_\lambda(x)=-U_\lambda(x^\lambda)$ and $V_\lambda(x)=-V_\lambda(x^\lambda).$

Now,  we will prove the maximum principles for the logarithmic Laplacian.
\begin{theorem} \label{le4} (Maximum principle for anti-symmetric functions) Let $\Omega$ be  a bounded Lipschitz domain in $\Sigma_{\lambda}.$  Suppose that $U_{\lambda} \in  L_{0}^{1}(\Omega)$ and $V_{\lambda} \in  L_{0}^{1}(\Omega)$ are Dini continuous functions in $\Omega$ and  continuous functions in $\bar{\Omega}$ such that the pair $(U_\lambda, V_\lambda)$ satisfies 
\begin{equation}\label{a5}
\begin{cases}
\ \mathcal{L}_{\Delta}U_{\lambda}(x)\geq0 ,& in\ \Omega,\\
\ \mathcal{L}_{\Delta}V_{\lambda}(x)\geq0 ,& in\ \Omega,\\
\ U_{\lambda}(x)\geq0 , V_{\lambda}(x)\geq0,& in\ \Sigma_{\lambda}\backslash \Omega,\\
\ U_{\lambda}(x^{\lambda})=-U(x), & in\  \Sigma_{\lambda},\\
\ V_{\lambda}(x^{\lambda})=-V(x), & in\  \Sigma_{\lambda}.
\end{cases}
\end{equation}
Then we have 
\begin{equation}\label{a6}
U_{\lambda}(x)\geq0,\ and\ V_{\lambda}(x)\geq0,\quad in\ \Omega.
\end{equation}

If $U_{\lambda}(x)$ and $V_{\lambda}(x)$ attain $0$ somewhere in
$\Omega$, then
$$U_{\lambda}(x)=V_{\lambda}(x)\equiv0,\quad a.e.\  \text{for}\ x\in\ \mathbb{R}^{n}.$$

\end{theorem}
\begin{proof}
Suppose, if possible, that (\ref{a6}) does not hold, so that there exists a $x^{o}\in\bar{\Omega}$ such that
$U_{\lambda}(x^{o})=\min\limits_{x\in\bar{\Omega}}U_{\lambda}(x)<0\,\,\text{or}\,\,
V_{\lambda}(x^{o})=\min\limits_{x\in\bar{\Omega}}V_{\lambda}(x)<0.$
We  can further deduce from (\ref{a5}) that $x^{o}$ is in the interior of $\Omega$. If $U_{\lambda}(x^{o})<0,$ it follows from \eqref{a1} that

\begin{align}\label{a7} \nonumber
\mathcal{L}_{\Delta}U_{\lambda}(x^{o})&=C_{n}\int_{\Sigma_{\lambda}}\frac{U_{\lambda}(x^{o})1_{B_{1}(x^{o})}(y)-U_{\lambda}(y)}{|x^{o}-y|^{n}}dy \\& \quad\quad\quad
+C_{n}\int_{\mathbb{R}^n \backslash \Sigma_{\lambda}}\frac{U_{\lambda}(x^{o})1_{B_{1}(x^{o})}(y)-U_{\lambda}(y)}{|x^{o}-y|^{n}}dy+\rho_{n}U_{\lambda}(x^{o})\\
&=C_{n}\int_{\Sigma_{\lambda}}\frac{U_{\lambda}(x^{o})1_{B_{1}(x^{o})}(y)-U_{\lambda}(y)}{|x^{o}-y|^{n}}dy \nonumber\\& \quad\quad\quad
+C_{n}\int_{ \Sigma_{\lambda}}\frac{U_{\lambda}(x^{o})1_{B_{1}(x^{o})}(y^\lambda)-U_{\lambda}(y^\lambda)}{|x^{o}-y^\lambda|^{n}}dy+\rho_{n}U_{\lambda}(x^{o})\nonumber\\ 
&\leq C_{n}\int_{\Sigma_{\lambda}}\frac{U_{\lambda}(x^{o})1_{B_{1}(x^{o})}(y)-U_{\lambda}(y)}{|x^{o}-y^\lambda|^{n}}dy \nonumber \\& \quad\quad\quad
+C_{n}\int_{ \Sigma_{\lambda}}\frac{U_{\lambda}(x^{o})1_{B_{1}(x^{o})}(y^\lambda)+U_{\lambda}(y)}{|x^{o}-y^\lambda|^{n}}dy+\rho_{n}U_{\lambda}(x^{o})\nonumber\\
&\leq C_{n}\int_{\Sigma_{\lambda}}\frac{U_{\lambda}(x^{o}) (1_{B_{1}(x^{o})}(y)+1_{B_{1}(x^{o})}(y^\lambda))}{|x^{o}-y^{\lambda}|^{n}}dy+\rho_{n}U_{\lambda}(x^{o})<0, \nonumber
\end{align}
which contradicts the inequality (\ref{a5}).

Similarly, for $V_{\lambda}(x^{o})<0$, we also reach at contradiction with (\ref{a5}). This verifies (\ref{a6}).

Now, we have proved that $U_{\lambda}(x)\geq0$, and $V_{\lambda}(x)\geq0$ in $\Sigma_{\lambda}$. If there exists some point $x^{o}\in\Omega$, such that $U_{\lambda}(x^{o})=0$ and $V_{\lambda}(x^{o})=0$, then from (\ref{a7}) we obtain
\begin{equation}\label{a8}
\begin{aligned}
0&\leq\mathcal{L}_{\Delta}U_{\lambda}(x^{o})=C_{n}\int_{\Sigma_{\lambda}}\frac{-U_{\lambda}(y)}{|x^{o}-y|^{n}}dy
-C_{n}\int_{\mathbb{R}^{n}\backslash \Sigma_{\lambda}}\frac{U_{\lambda}(y)}{|x^{o}-y|^{n}}dy=C_{n}\int_{\mathbb{R}^{n}}\frac{-U_{\lambda}(y)}{|x^{o}-y|^{n}}dy,
\end{aligned}
\end{equation}
and
\begin{equation}\label{a9}
0\leq\mathcal{L}_{\Delta}V_{\lambda}(x^{o})=C_{n}\int_{\mathbb{R}^{n}}\frac{-V_{\lambda}(y)}{|x^{o}-y|^{n}}dy.
\end{equation}
We derive immediately using $U_\lambda\geq 0$ and $V_\lambda\geq 0$ that
$$U_{\lambda}(x)=V_{\lambda}(x)\equiv0,\quad a.e.\ in\ \mathbb{R}^{n}.$$

This completes the proof.
\end{proof}

The following narrow region principle will  also be necessary for proving subsequent results.
\begin{theorem} \label{le5} (Narrow region principle)
Let $\Omega$ be a bounded (narrow) Lipschitz domain in $\Sigma_{\lambda}$ such that it is contained in $\{x \in \mathbb{R}^n\mid\lambda-l<x_{1}<\lambda\}$ with small $l.$  Suppose that $c_{i}(x),i=1,2$, are both bounded from below in $\Omega$ with $c_2(x) \leq 0$. Suppose that $U_{\lambda} \in  L_{0}^{1}(\Omega)$ and $V_{\lambda} \in  L_{0}^{1}(\Omega)$ are Dini continuous functions in $\Omega$ and  continuous functions in $\bar{\Omega}$ such that they satisfy 
\begin{equation}\label{a10}
\begin{cases}
\ \mathcal{L}_{\Delta}U_{\lambda}(x)+c_{1}(x)V_{\lambda}(x)\geq0 ,& in\ \Omega,\\
\ \mathcal{L}_{\Delta}V_{\lambda}(x)+c_{2}(x)U_{\lambda}(x)\geq0 ,& in\ \Omega,\\
\ U_{\lambda}(x)\geq0 ,\  V_{\lambda}(x)\geq0,& in\ \Sigma_{\lambda}\backslash \Omega,\\
\ U_{\lambda}(x^{\lambda})=-U(x), V_{\lambda}(x^{\lambda})=-V(x), & in\  \Sigma_{\lambda},
\end{cases}
\end{equation}
then, for sufficiently small $l$, we have
\begin{equation}\label{a11}
U_{\lambda}(x)\geq0,\ and\ V_{\lambda}(x)\geq0,\quad in\ \Omega.
\end{equation}

Furthermore, if $U_{\lambda}(x)$ and $V_{\lambda}(x)$ attain $0$ somewhere in
$\Omega$, then
\begin{equation}
    U_{\lambda}(x)=V_{\lambda}(x)\equiv0,\quad a.e.\ in\ \mathbb{R}^{n}.
\end{equation}
\end{theorem}

\begin{proof}
If (\ref{a11}) does not hold then there exists a point $x^{o}\in\bar{\Omega}$ such that
$$U_{\lambda}(x^{o})=\min_{x\in\bar{\Omega}}U_{\lambda}(x)<0\,\,\text{or}\,\, V_{\lambda}(x^{o})=\min_{x\in\bar{\Omega}}V_{\lambda}(x)<0,$$
and one can further deduce using the assumption $U_\lambda \geq  0, V_\lambda \geq 0$ on $\Sigma_{\lambda} \backslash \Omega$ from (\ref{a5}) that $x^{o}$ is in the interior of $\Omega$.

Without loss of generality, we suppose that $U_{\lambda}(x^{o})<0$. Let us first consider the case when $V_{\lambda}(x^{o})\geq0$. By performing similar calculations as in (\ref{a7}), we get
\begin{equation}\label{a12}
\begin{aligned}
\mathcal{L}_{\Delta}U_{\lambda}(x^{o})\leq C_{n}\int_{\Sigma_{\lambda}}\frac{U_{\lambda}(x^{o}) (1_{B_{1}(x^{o})}(y)+1_{B_{1}(x^{o})}(y^\lambda))}{|x^{o}-y^{\lambda}|^{n}}dy+\rho_{n}U_{\lambda}(x^{o}).
\end{aligned}
\end{equation}

Let $H=\{y\mid l<y_{1}-x_{1}^{o}<1, |y'-(x^{o})'|<1\}$, $s=y_{1}-x_{1}^{o}$, $\tau=|y'-(x^{o})'|$ and $\omega_{n-2}=|B_{1}(0)|$ in $\mathbb{R}^{n-2}$. Now we have
\begin{equation}\label{a13}
\begin{aligned}
\int_{\Sigma_{\lambda}}\frac{1}{|x^{o}-y^{\lambda}|^{n}}dy&\geq\int_{H}\frac{1}{|x^{o}-y^{\lambda}|^{n}}dy=\int_{l}^{1}\int_{0}^{1}\frac{\omega_{n-2}\tau^{n-2}}{(s^{2}+\tau^{2})^{\frac{n}{2}}}d\tau ds\\
&\geq\int_{l}^{1}\frac{1}{s}\int_{0}^{1}\frac{\omega_{n-2}t^{n-2}}{(1+t^{2})^{\frac{n}{2}}}dt ds\geq C\int_{l}^{1}\frac{1}{s}ds\rightarrow\infty,\,\,\,\text{as}\,\,l\rightarrow0,
\end{aligned}
\end{equation}
where we have used the substitution $\tau=st$.

Since $c_{1}(x)$ is lower bounded in $\Omega$ and $U_\lambda(x^{o})<0$, we get, using   \eqref{a13}, that
\begin{equation}\label{a14}
\begin{aligned}
\mathcal{L}_{\Delta}U_{\lambda}(x^{o})+c_{1}(x^{o})V_{\lambda}(x^{o})&\leq CU_{\lambda}(x^{o})\int_{l}^{1}\frac{1}{s}ds+c_{1}(x^{o})V_{\lambda}(x^{o})\\
&=U_{\lambda}(x^{o})\Bigg(C\int_{l}^{1}\frac{1}{s}ds+c_{1}(x^{o})\frac{V_{\lambda}(x^{o})}{U_{\lambda}(x^{o})}\Bigg)<0,
\end{aligned}
\end{equation}
when $l$ sufficiently small. This is contradiction for a condition in (\ref{a10}). Therefore, (\ref{a11}) must be true.

Now, let us consider the case when $V_{\lambda}(x^{o})<0$, by (\ref{a10}) and (\ref{a13}), we have
\begin{equation}\label{a15}
\begin{aligned}
0\leq\mathcal{L}_{\Delta}U_{\lambda}(x^{o})+c_{1}(x^{o})V_{\lambda}(x^{o})\leq CU_{\lambda}(x^{o})\int_{l}^{1}\frac{1}{s}ds+c_{1}(x^{o})V_{\lambda}(x^{o}).
\end{aligned}
\end{equation}
Since $V_{\lambda}(x^{o})<0$, there exists $\hat{x}\in\Omega$ such that
$$V_{\lambda}(\hat{x})=\min_{\bar{\Omega}}V_{\lambda}(x)<0.$$
Similarly to (\ref{a15}), by (\ref{a10}) and $c_{2}(x)\leq0$, we get
\begin{equation}\label{a16}
\begin{aligned}
\mathcal{L}_{\Delta}V_{\lambda}(\hat{x})+c_{2}(\hat{x})U_{\lambda}(\hat{x})&\leq CV_{\lambda}(\hat{x})\int_{l}^{1}\frac{1}{s}ds+c_{2}(\hat{x})U_{\lambda}(\hat{x})\\
&\leq CV_{\lambda}(\hat{x})\int_{l}^{1}\frac{1}{s}ds+c_{2}(\hat{x})U_{\lambda}(x^{o}).
\end{aligned}
\end{equation}

Adding (\ref{a15}) and (\ref{a16}), we get
\begin{equation}\label{a17}
\begin{aligned}
\Bigg(C\int_{l}^{1}\frac{1}{s}ds+c_{2}(\hat{x})\Bigg)U_{\lambda}(x^{o})+
\Bigg(C\int_{l}^{1}\frac{1}{s}ds+c_{1}(x^{o})\Bigg)V_{\lambda}(\hat{x})\geq0.
\end{aligned}
\end{equation}
Since $U_{\lambda}(x^{o})<0$ and $V_{\lambda}(\hat{x})<0$, then at least one of the following holds:
\begin{equation}\label{a18}
C\int_{l}^{1}\frac{1}{s}ds+c_{2}(\hat{x})\leq0,\ or \ C\int_{l}^{1}\frac{1}{s}ds+c_{1}(x^{o})\leq0.
\end{equation}
However, when $l$ sufficiently small, from the fact that $c_{i}(x),i=1,2$ are both bounded from below in $\Omega$, we obtain
\begin{equation}\label{a19}
C\int_{l}^{1}\frac{1}{s}ds+c_{2}(\hat{x})>0,\ and \ C\int_{l}^{1}\frac{1}{s}ds+c_{1}(x^{o})>0,
\end{equation}
which is a contradiction with (\ref{a18}).
Therefore, (\ref{a11}) must be true. 

This completes the proof.
\end{proof}

The following decay at infinity is also indispensable for further use in the main result.

\begin{theorem} \label{le6} (Decay at infinity)
Assume that $\Omega$ is an unbounded Lipschitz domain in $\Sigma_{\lambda}$. Suppose that $U_{\lambda} \in  L_{0}^{1}(\Omega)$ and $V_{\lambda} \in  L_{0}^{1}(\Omega)$ are Dini continuous functions in $\Omega$ and  continuous functions in $\bar{\Omega}$ such that the pair $(U_\lambda, V_\lambda)$ satisfies 
\begin{equation}\label{a20}
\begin{cases}
\ \mathcal{L}_{\Delta}U_{\lambda}(x)+c_{1}(x)V_{\lambda}(x)\geq0 ,& in\ \Omega,\\
\ \mathcal{L}_{\Delta}V_{\lambda}(x)+c_{2}(x)U_{\lambda}(x)\geq0 ,& in\ \Omega,\\
\ U_{\lambda}(x)\geq0 ,\  V_{\lambda}(x)\geq0,& in\ \Sigma_{\lambda}\backslash \Omega,\\
\ U_{\lambda}(x^{\lambda})=-U(x), V_{\lambda}(x^{\lambda})=-V(x), & in\  \Sigma_{\lambda},
\end{cases}
\end{equation}
with
$c_{i}(x)\leq0\,\,\,\text{for all}\,\,x \in \Omega, $
and
\begin{equation}\label{a21}
c_{i}(x)= O\bigg(\frac{1}{|x|^{\kappa_{i}}}\bigg),\,\,\text{for some}\,\,k_i>0,\,\,\text{as}\,\, |x|\rightarrow \infty,\,\,\ \text{for}\, \ i=1,2.
\end{equation}
Then, there exists a constant $R_{o}>0$ such that if
$$U_{\lambda}(x^{o})=\min_{x\in\bar{\Omega}}U_{\lambda}(x)<0,\ or
\ V_{\lambda}(x^{o})=\min_{x\in\bar{\Omega}}V_{\lambda}(x)<0,$$
then
\begin{equation}\label{a22}
|x^{o}|\leq R_{o}.
\end{equation}

\end{theorem}

\begin{proof}
Without loss of generality, we suppose that $U_{\lambda}(x^{o})<0$. If $V_{\lambda}(x^{o})\geq0$, similarly to (\ref{a7}), we have 
\begin{equation}\label{a23}
\begin{aligned}
\mathcal{L}_{\Delta}U_{\lambda}(x^{o})\leq C_{n}\int_{\Sigma_{\lambda}}\frac{U_{\lambda}(x^{o}) (1_{B_{1}(x^{o})}(y)+1_{B_{1}(x^{o})}(y^\lambda))}{|x^{o}-y^{\lambda}|^{n}}dy+\rho_{n}U_{\lambda}(x^{o}).
\end{aligned}
\end{equation}

Suppose that \eqref{a22} does not hold. Then, for each fixed $\lambda$, when $|x^{o}|\geq\lambda$, we have $B_{|x^{o}|}(x^{1})\subset\mathbb{R}^{n}\backslash\Sigma_{\lambda}$ with $x^{1}=(3|x^{o}|+x_{1}^{o},(x^{o})')$, and it follows that
\begin{equation}\label{a24}
\begin{aligned}
\int_{\Sigma_{\lambda}}\frac{(1_{B_{1}(x^{o})}(y)+1_{B_{1}(x^{o})}(y^\lambda))}{|x^{o}-y^{\lambda}|^{n}}dy&\geq\int_{B_{|x^{o}|}(x^{1})}\frac{1}{|x^{o}-y|^{n}}dy\geq\int_{B_{|x^{o}|}(x^{1})}\frac{1}{4^{n}|x^{o}|^{n}}dy=\frac{\omega_{n}}{4^{n}},
\end{aligned}
\end{equation}
where (\ref{a24}) follows from $|x^{o}-y|\leq|x^{o}-x_{1}|+|x^{o}|=4|x^{o}|$ for all $y\in B_{|x^{o}|}(x^{1})$. 

Then we have from \eqref{a23} and \eqref{a24} using the fact that $U_\lambda (x^o)<0$ that
\begin{equation}\label{a25}
\begin{aligned}
0&\leq \mathcal{L}_{\Delta}U_{\lambda}(x^{o})+c_{1}(x^{o})V_{\lambda}(x^{o})\\
&\leq \left(\frac{C_n\omega_{n}}{4^{n}}+\rho_{n}\right)U_{\lambda}(x^{o})+c_{1}(x^{o})V_{\lambda}(x^{o}).
\end{aligned}
\end{equation}

Since  $c_1(x^{o})\leq 0$  from (\ref{a21}), we have
\begin{equation}\label{a26}
\left(\frac{C_n\omega_{n}}{4^{n}}+\rho_{n}\right)U_{\lambda}(x^{o})+c_{1}(x^{o})V_{\lambda}(x^{o})<0,
\end{equation}
which contradicts  \eqref{a25}.

If $V_{\lambda}(x^{o})<0$, we obtain
\begin{equation}\label{a27}
\begin{aligned}
0&\leq \mathcal{L}_{\Delta}U_{\lambda}(x^{o})+c_{1}(x^{o})V_{\lambda}(x^{o})\\
&\leq \left(\frac{C_n\omega_{n}}{4^{n}}+\rho_{n}\right)U_{\lambda}(x^{o})+c_{1}(x^{o})V_{\lambda}(x^{o}),
\end{aligned}
\end{equation}
and
\begin{equation}\label{a28}
\begin{aligned}
0&\leq \mathcal{L}_{\Delta}V_{\lambda}(x^{o})+c_{2}(x^{o})U_{\lambda}(x^{o})\\
&\leq \left(\frac{C_n\omega_{n}}{4^{n}}+\rho_{n}\right)V_{\lambda}(x^{o})+c_{2}(x^{o})U_{\lambda}(x^{o}).
\end{aligned}
\end{equation}

Adding (\ref{a27}) and (\ref{a28}), we get
\begin{equation}\label{a29}
\begin{aligned}
\left(\frac{C_n\omega_{n}}{4^{n}}+\rho_{n}+c_{2}(x^{o}) \right)U_{\lambda}(x^{o})
+\left(\frac{C_n\omega_{n}}{4^{n}}+\rho_{n}+c_{1}(x^{o}) \right)V_{\lambda}(x^{o})\geq0.
\end{aligned}
\end{equation}

Since $U_{\lambda}(x^{o})<0$ and $V_{\lambda}(x^{o})<0$, if (\ref{a29}) holds, at least one of
\begin{equation}\label{a30}
\frac{C_n\omega_{n}}{4^{n}}+\rho_{n}+c_{2}(x^{o})\leq0, \ or\
\frac{C_n\omega_{n}}{4^{n}}+\rho_{n}+c_{1}(x^{o})\leq0,
\end{equation}
holds. However, if $|x^{o}|$ is sufficiently large, it follows from (\ref{a21})   that
$$\frac{C_n\omega_{n}}{4^{n}}+\rho_{n}+c_{i}(x^{o})>0,\ i=1,2,$$
which is a contradiction with \eqref{a30}.

Thus, we get a contradiction in the case when $V_\lambda(x^{o})\geq 0$ as well as when $V_\lambda(x^{o})\leq 0.$ Therefore, we conclude that (\ref{a22}) holds. 

This completes the proof.
\end{proof}

\section{The proof of the main results}
In this section, we will present the proof of the main results.

\begin{proof}[Proof of Theorem \ref{th1}]
Let
$U_{\lambda}(x)=u_{\lambda}(x)-u(x),\quad V_{\lambda}(x)=v_{\lambda}(x)-v(x).$

$\mathbf{Step\ 1.}$ We show that, for $\lambda$ sufficiently negative,
\begin{equation}\label{b3}
U_{\lambda}(x)\geq0\ and\ V_{\lambda}(x)\geq0,\quad x\in\Sigma_{\lambda}.
\end{equation}

First, notice that, by the definition of $U_{\lambda}$ and $V_{\lambda}$, we have
\begin{equation}\label{b4}
\lim_{|x|\rightarrow\infty}U_{\lambda}(x)=0,\quad \lim_{|x|\rightarrow\infty}V_{\lambda}(x)=0.
\end{equation}
Thus, if $U_{\lambda}$ or $V_{\lambda}$ is negative somewhere in $\Sigma_{\lambda}$, then the negative minima of $U_{\lambda}$ or $V_{\lambda}$ would be attained in the interior of $\Sigma_{\lambda}$.

From \eqref{c1}, at points where $V_{\lambda}(x)$ is negative, we have
\begin{equation}\label{b5}
\begin{aligned}
\mathcal{L}_{\Delta}U_{\lambda}(x)&=v_{\lambda}^{p}(x)-v^{p}(x)\\
&\geq p v^{p-1}(x)V_{\lambda}(x),
\end{aligned}
\end{equation}
where \eqref{b5} follows from the mean value theorem. Therefore, we get
\begin{equation}\label{b6}
\begin{aligned}
\mathcal{L}_{\Delta}U_{\lambda}(x)+c_{1}(x)V_{\lambda}(x)\geq0,
\end{aligned}
\end{equation}
with
\begin{equation}\label{b7}
c_{1}(x)=-p v^{p-1}(x).
\end{equation}
Likewise, when $U_{\lambda}(x)$ is negative at some points, we similarly have
\begin{equation}\label{b8}
\begin{aligned}
\mathcal{L}_{\Delta}V_{\lambda}(x)+c_{2}(x)U_{\lambda}(x)\geq0,
\end{aligned}
\end{equation}
with
\begin{equation}\label{b9}
c_{2}(x)=-q u^{q-1}(x).
\end{equation}
It is easy to verify using \eqref{b2}, \eqref{b3} and \eqref{b7} that, for $|x|$ sufficiently large,
\begin{equation}\label{b10}
c_{1}(x)=O\left(\frac{1}{|x|^{\beta(p-1)}}\right).
\end{equation}
In the same way, we obtain, for $|x|$ sufficiently large, that
\begin{equation}\label{b11}
c_{2}(x)=O\left(\frac{1}{|x|^{\alpha(q-1)}}\right).
\end{equation}

Moreover, it follows from \eqref{b6} and \eqref{b8} that $c_{i}(x)\leq0,i=1,2.$ Hence, $c_{i}(x)$ satisfy \eqref{a21} in Theorem \ref{le6}.
Therefore, we conclude using Theorem \ref{le6} that, there exists an $R_{o}>0$ (independent of $\lambda$), such that if $\bar{x}$ is a negative minimum of $U_{\lambda}$ or $V_{\lambda}$ in $\Sigma_{\lambda}$, then
\begin{equation}\label{b12}
|\bar{x}|\leq R_{o}.
\end{equation}

Now, for $\lambda\leq-R_{o}$, we must have
$$U_{\lambda}(x)\geq0,\ and\ V_{\lambda}(x)\geq0,\quad x\in\Sigma_{\lambda}.$$
$\mathbf{Step\ 2.}$
Step 1 provides a starting point, from which we can now move the plane $T_{\lambda}$ to the right as long as \eqref{b3} holds to its limiting position.

Let us define the following 
$$\lambda_{0}=\sup\{\lambda<x_{1}^{o}\,\mid U_{\mu}(x)\geq0,\ V_{\mu}(x)\geq0,\ \forall x\in \Sigma_{\mu},\  \mu\leq\lambda\}.$$
In this part, our aim is to show that
$$\lambda_{0}=x_{1}^{o},$$
and
\begin{equation}\label{b13}
U_{\lambda_{0}}(x)\equiv0,\ and\ \ V_{\lambda_{0}}(x)\equiv0,\quad x\in\Sigma_{\lambda_{0}}.
\end{equation}

Suppose that $\lambda_{0}<x_{1}^{o}$, we show that the plane $T_{\lambda}$ can be moved further right, that is, there exists some $\epsilon>0$, such that for any $\lambda\in(\lambda_{0},\lambda_{0}+\epsilon)$, we have
\begin{equation}\label{b14}
U_{\lambda}(x)\geq0,\ and\ V_{\lambda}(x)\geq0,\quad x\in\Sigma_{\lambda}.
\end{equation}
This would be a contradiction with the definition of $\lambda_{0}$. Therefore, we must have
\begin{equation}\label{b15}
\lambda_{0}=x_{1}^{o}.
\end{equation}

Now we prove \eqref{b13} by combining the use of the narrow region principle (Theorem \ref{le5}) and decay at infinity (Theorem \ref{le6}).

By \eqref{b12}, the negative minimum of $U_{\lambda}(x)$ cannot be attained outside of $B_{R_{o}}(0)$. Now, we will show that it can not be attained inside of $B_{R_{o}}(0)$ as well. That is, we will show that for $\lambda$ sufficiently close to $\lambda_{0}$,
\begin{equation}\label{b16}
U_{\lambda}(x)\geq0,\ and\ V_{\lambda}(x)\geq0,\quad x\in\Sigma_{\lambda}\cap B_{R_{o}(0)}.
\end{equation}

According to Theorem \ref{le5}, there is a small $\delta>0$, such that for $\lambda\in[\lambda_{0},\lambda_{0}+\delta)$, if
\begin{equation}\label{b17}
U_{\lambda}(x)\geq0,\ and\ V_{\lambda}(x)\geq0,\quad x\in\Sigma_{\lambda_{0}-\delta},
\end{equation}
then
\begin{equation}\label{b18}
U_{\lambda}(x)\geq0,\ and\ V_{\lambda}(x)\geq0,\quad x\in\Sigma_{\lambda}\backslash \Sigma_{\lambda_{0}-\delta}.
\end{equation}

Then what is left is to show \eqref{b17}, and actually, we only need
\begin{equation}\label{b19}
U_{\lambda}(x)\geq0,\ and\ V_{\lambda}(x)\geq0,\quad x\in\Sigma_{\lambda_{0}-\delta}\cap B_{R_{o}(0)}.
\end{equation}

In fact, when $\lambda_{0}<x_{1}^{o}$, we have
\begin{equation}\label{b20}
U_{\lambda_{0}}(x)>0,\ and\ V_{\lambda_{0}}(x)>0,\quad x\in\Sigma_{\lambda_{0}}.
\end{equation}
If not, there exists some $\tilde{x}$ such that
$$U_{\lambda_{0}}(\tilde{x})=\min_{\Sigma_{\lambda_{0}}}U_{\lambda_{0}}(x)=0,\,\,\text{or}\,\,
\ V_{\lambda_{0}}(\tilde{x})=\min_{\Sigma_{\lambda_{0}}}V_{\lambda_{0}}(x)=0.$$

Without loss of generality, we suppose that $$U_{\lambda_{0}}(\tilde{x})=\min_{\Sigma_{\lambda_{0}}}U_{\lambda_{0}}(x)=0.$$ 
For $V_{\lambda_{0}}(\tilde{x})\geq0$, since $|\tilde{x}-y|>|\tilde{x}-y^{\lambda_{0}}|$, it follows that
\begin{equation}\label{b21}
\begin{aligned}
\mathcal{L}_{\Delta}U_{\lambda_{0}}(\tilde{x})&=C_{n}\int_{\Sigma_{\lambda_{0}}}\frac{-U_{\lambda_{0}}(y)}{|\tilde{x}-y|^{n}}dy
-C_{n}\int_{\mathbb{R}^{n}\backslash \Sigma_{\lambda_{0}}}\frac{U_{\lambda_{0}}(y)}{|\tilde{x}-y|^{n}}dy\\
&=C_{n}\int_{\Sigma_{\lambda_{0}}}\Big(\frac{1}{|\tilde{x}-y^{\lambda_{0}}|^{n}}-\frac{1}{|\tilde{x}-y|^{n}}\Big)U_{\lambda_{0}}(y)dy\\
&<0.
\end{aligned}
\end{equation}
On the other hand,
\begin{equation}\label{b22}
\begin{aligned}
\mathcal{L}_{\Delta}U_{\lambda_{0}}(\tilde{x})&=v_{\lambda_{0}}^{p}(\tilde{x})-v^{p}(\tilde{x})\\
&\geq pv^{p-1}(\tilde{x})V_{\lambda_{0}}(\tilde{x})\\
&\geq0,
\end{aligned}
\end{equation}
which is a contradiction with \eqref{b21}.

Therefore, we get \eqref{b20}. It follows from \eqref{b20} and Theorem \ref{le4} that there exists a constant $c_{0}>0$, such that
$$U_{\lambda_{0}}(x)\geq c_{0},\ and\ V_{\lambda_{0}}(x)\geq c_{0},\quad x\in \overline{\Sigma_{\lambda_{0}-\delta}\cap B_{R_{o}(0)}}.$$
Since $U_{\lambda}$ and $V_{\lambda}$ both depend on $\lambda$ continuously, there exists $0<\epsilon<\delta$, such that for all $\lambda\in(\lambda_{0},\lambda_{0}+\epsilon)$, we have
\begin{equation}\label{b23}
U_{\lambda}(x)\geq c_{0},\ and\ V_{\lambda}(x)\geq c_{0},\quad x\in \overline{\Sigma_{\lambda-\delta}\cap B_{R_{o}(0)}}.
\end{equation}

Combining \eqref{b11}, \eqref{b18} and \eqref{b23}, we conclude that for all $\lambda\in(\lambda_{0},\lambda_{0}+\epsilon)$,
$$U_{\lambda}(x)\geq 0,\ and\ V_{\lambda}(x)\geq 0,\quad x\in\Sigma_{\lambda}.$$
This contradicts the definition of $\lambda_{0}$, Therefore, we must have
$$\lambda_{0}=x_{1}^{o},\ \ and\ \ U_{\lambda_{0}}(x)\geq 0,\ and\ V_{\lambda_{0}}(x)\geq 0,\quad x\in\Sigma_{\lambda_{0}}.$$

Similarly, one can move the plane $T_{\lambda}$ from $+\infty$ to the left and show that
$$U_{\lambda_{0}}(x)\leq 0,\ and\ V_{\lambda_{0}}(x)\leq 0,\quad x\in\Sigma_{\lambda_{0}}.$$

Thus,  we have shown that
$$\lambda_{0}=x_{1}^{o},\ \ and\ \ U_{\lambda_{0}}(x)\equiv0,\ and\ V_{\lambda_{0}}(x)\equiv0,\quad x\in\Sigma_{\lambda_{0}},$$
completing the proof of {\bf Step 2}.

Since the $x_{1}$-direction can be chosen arbitrarily, we have actually shown that $(u,v)$
are radially symmetric about $x^{o}$. The monotonicity of $(u,v)$ follows from the process of moving the planes.

This is complete the proof of Theorem \ref{th1}.
\end{proof}
For the convenience of the reader, we present a concise proof of Theorem \ref{thm} as it proceeds similar to the proof of Theorem \ref{th1}.

\begin{proof}[Proof of Theorem \ref{thm}] The proof will again be based on  the moving planes method. For this, purpose, 
let us define
$$W_{i,\lambda}(x):=u_{i,\lambda}(x)-u_{i}(x),$$
where $u_{i,\lambda}(x)=u_{i}(x^{\lambda})$, and $i=1,2,\ldots, m$.

$\mathbf{Step\ 1.}$ We show that, for $\lambda$ sufficiently negative,
\begin{equation*}\label{w1}
W_{i,\lambda}(x)\geq0,\quad \forall x\in\Sigma_{\lambda}\quad \forall \,i=1,2,\ldots, m.
\end{equation*}

By the definition of $W_{i,\lambda}(x),\, i=1,2,\ldots, m$, we conclude
\begin{equation}\label{w2}
\lim_{|x|\rightarrow\infty}W_{i,\lambda}(x)=0.
\end{equation}

From \eqref{m}, at points where $W_{j+1,\lambda}(x)$ is negative, we have
\begin{equation*}\label{w3}
\begin{aligned}
\mathcal{L}_{\Delta}W_{j,\lambda}(x)&=u_{j+1,\lambda}^{p_{j}}(x)-u_{j+1}^{p_{j}}(x)\\
&\geq p_{j}u_{j+1}^{p_{j}-1}(x)W_{j+1,\lambda}(x),
\end{aligned}
\end{equation*}
for $j=1,\cdot\cdot\cdot,m-1$, and 
\begin{equation}\label{w4}
\mathcal{L}_{\Delta}W_{m,\lambda}(x)\geq p_{m}u_{1}^{p_{m}-1}(x)W_{1,\lambda}(x).
\end{equation}

Therefore, for $j=1,\cdot\cdot\cdot,m-1$, we obtain
\begin{equation*}\label{w5}
\begin{aligned}
\mathcal{L}_{\Delta}W_{j,\lambda}(x)+c_{j}(x)W_{j+1,\lambda}(x)\geq0,
\end{aligned}
\end{equation*}
with
\begin{equation}\label{w6}
c_{j}(x)=-p_{j}u_{j+1}^{p_{j}-1}(x),
\end{equation}
and
\begin{equation*}\label{w7}
\begin{aligned}
\mathcal{L}_{\Delta}W_{m,\lambda}(x)+c_{m}(x)W_{1,\lambda}(x)\geq0,
\end{aligned}
\end{equation*}
with
\begin{equation}\label{w8}
c_{m}(x)=-p_{m}u_{1}^{p_{m}-1}(x).
\end{equation}

It is easy to verify that, for $|x|$ sufficiently large, we have
\begin{equation*}\label{w19}
c_{j }(x)=O\left(\frac{1}{|x|^{\alpha_{j+1}(p_{j }-1)}}\right), \quad \text{for}\,\,\,\,j=1,\ldots,m-1,
\end{equation*}
and
\begin{equation*}\label{w9}
c_{m}(x)=O\left(\frac{1}{|x|^{\alpha_{1}(p_{m }-1)}}\right).
\end{equation*}

Moreover, it follows from \eqref{w6} and \eqref{w8} that $c_{i}(x)\leq0$.
Therefore, we conclude using Theorem \ref{le6} that, there exists an $R_{o}>0$ (independent of $\lambda$), such that, if $\bar{x}$ is a negative minimum of $W_{i,\lambda}(x)$ in $\Sigma_{\lambda}$, then
\begin{equation}\label{w10}
|\bar{x}|\leq R_{o}.
\end{equation}
$\mathbf{Step\ 2.}$
Step 1 provides a starting point, from which we can now move the plane $T_{\lambda}$ to the right as long as \eqref{w1} holds to its limiting position.

Let us define the following 
$$\lambda'_{0}=\sup\{\lambda<x_{1}^{o}\,\mid W_{i,\mu}(x)\geq0,\,\text{for}\,\,\,i=1,\ldots,m,\ \forall x\in \Sigma_{\mu},\  \mu\leq\lambda\}.$$
Similar to the proof of Step 2 in Theorem \ref{th1}, we can get
$$\lambda'_{0}=x_{1}^{o},$$
and
\begin{equation}\label{w11}
W_{i,\lambda'_{0}}(x)\equiv0,\ x\in\Sigma_{\lambda'_{0}},
\end{equation}
completing the proof of Theorem \ref{thm}.
\end{proof}

\section*{Acknowledgments}
RZ is supported by the China Scholarship Council (No.202106860018), the National Natural Science Foundation of China (Grant No. 11871278) and by the National Natural Science
Foundation of China (Grant No. 11571093).
 VK and MR are supported by the FWO Odysseus 1 grant G.0H94.18N: Analysis and Partial
Differential Equations, the Methusalem programme of the Ghent University Special Research Fund (BOF) (Grant number 01M01021) and by FWO Senior Research Grant G011522N. MR is also supported by EPSRC grant
EP/R003025/2.


\begin{thebibliography}{99}

\bibitem{chen4}
C. Brandle, E. Colorado, A. de Pablo, U. Sanchez, A concave-convex elliptic problem involving the
fractional Laplacian, Proc. Roy. Soc. Edinburgh Sect. A 143 (2013): 39-71.


\bibitem{t1}
 R. Ba$\tilde{\text{n}}$uelos, R. Latala, P. M\'endez-Hern\'andez, A Brascamp-Lieb-Luttinger-type inequality and
applications to symmetric stable processes. Proc. Am. Math. Soc., 129(10)(2001): 2997-3008.

\bibitem{t2}
 R. Ba$\tilde{\text{n}}$uelos, T. Kulczycki, The Cauchy process and the Steklov problem. J. Funct. Anal., 211(2)(2004):
355-423.

\bibitem{t3}
 G. Bisci, V. Radulescu, R. Servadei, Variational Methods for Nonlocal Fractional Problems,
vol. 162. Cambridge University Press, Cambridge (2016).

\bibitem{t4} K. Bogdan, T. Byczkowski, Potential theory for the $\alpha$-stable Schr\"odinger operator on bounded Lipschitz domains, Stud. Math., 133(1) (1999): 53-92.

\bibitem{t5}
K. Bogdan, T. Kulczycki, M. Kwa$\acute{s}$nicki, Estimates and structure of $\alpha$-harmonic functions, Probab.
Theory Related Fields, 140(3–4) (2008): 345-381.

\bibitem{t6}
 C. Bucur, E. Valdinoci, Nonlocal Diffusion and Applications, vol. 20. Springer, Cham (2016).

\bibitem{t7}
L. Caffarelli,  L. Silvestre, An extension problem related to the fractional Laplacian, Commun. Partial
Differ. Equ., 32(8) (2007): 1245-1260.
















\bibitem{cheng}
C. Cheng, Z. L\"u, Y. L\"u, A direct method of moving planes for the system of the fractional Laplacian, Pacific Journal of Mathematics, 2017, 290(2): 301-320.



\bibitem{chen}
H. Chen, T. Weth, The Dirichlet problem for the logarithmic Laplacian, Comm. Partial Differential Equations 44 (2019):
1100-1139.


\bibitem{a6}
L. Caffarelli, L. Silvestre, An extension problem related to the fractional Laplacian, Commun. Partial Differ. Equ. 32
(2007): 1245-1260.





\bibitem{chen3}
W. Chen, C. Li, Methods on Nonlinear Elliptic Equations, AIMS Book Series, vol. 4, 2010.


\bibitem{chen91}
W. Chen, C. Li, Classification of solutions of some nonlinear elliptic equations, Duke Math. J. 63
(1991): 615-622.


\bibitem{chen1}
W. Chen, C. Li, Y. Li, A direct method of moving planes for the fractional
Laplacian, Adv. Math., 308 (2017): 404-437.

\bibitem{chen2}
W. Chen, J. Zhu, Indefinite fractional elliptic problem and Liouville theorems, J. Differential Equations 260 (2016): 4758-4785.


\bibitem{cpa}
R. Frank, E. Lenzmann, L. Silvestre, Uniqueness of radial solutions for the fractional Laplacian,
Comm. Pure Appl. Math. 69 (2016): 1671-1726.

\bibitem{Fra2} R. L. Frank and E. Lenzmann, On ground states for the $L^2$-critical boson star equation. Preprint (2009). arXiv:0910.2721

\bibitem{Fra1} R. L. Frank, T. König and H. Tang, Classification of solutions of an equation related to a conformal log Sobolev inequality. Adv. Math. 375 (2020): 107395.

\bibitem{am}
P. Felmer, A. Quaas, Fundamental solutions and Liouville type theorems for nonlinear
integral operators, Adv. Math. 226(2011): 2712-2738.

\bibitem{FJW} P. A. Feulefack, S. Jarohs, T. Weth,  Small order asymptotics of the Dirichlet eigenvalue problem for the fractional Laplacian. J. Fourier Anal. Appl. 28 (2022), no. 2, Paper No. 18, 44 pp.

\bibitem{GNN} 
B. Gidas, W. Ni, L. Nirenberg, Symmetry and related properties via the maximum principle. Comm. Math. Phys. 68 (1979):  209-243. 

\bibitem{Guo}  Y. Guo, J. Liu,  Liouville type theorems for positive solutions of elliptic system in RN. Comm. Partial Differential Equations 33 (2008), no. 1-3, 263–284.

\bibitem{HVS} S. Hern\'andez, S. V\'ictor, A. Saldaña, Small order asymptotics for nonlinear fractional problems. Calc. Var. Partial Differential Equations 61 (2022), no. 3, Paper No. 92, 26 pp.

\bibitem{ampa}
S. Jarohs, T. Weth, Symmetry via antisymmetric maximum principles in nonlocal problems of
variable order, Ann. Mat. Pura Appl. 195 (2016): 273-291.


\bibitem{jems}
T. Jin, Y. Li, J. Xiong, On a fractional Nirenberg problem, part I: blow up analysis and compactness of solutions, J. Eur. Math. Soc. 16 (2014): 1111-1171.

\bibitem{LW} A. Laptev, T. Weth, Spectral properties of the logarithmic Laplacian. Anal. Math. Phys. 11 (2021), no. 3, Paper No. 133, 24 pp.

\bibitem{Liu18} B. Liu,  Direct method of moving planes for logarithmic Laplacian system in bounded domains. Discrete Contin. Dyn. Syst. 38 (2018), no. 10, 5339–5349.

\bibitem{lu}
G. Lu, J. Zhu, An overdetermined problem in Riesz-potential and fractional Laplacian, Nonlinear
Anal. 75 (2012): 3036-3048.

\bibitem{30}
E. Mitidieri, A Rellich type identity and applications, Comm. PDE, 18 (1993): 125-151.


\bibitem{31}
E. Mitidieri, Nonexistence of positive solutions of semilinear elliptic systems in $\mathbb{R}^{n}$, Diff.
Integ. Equa., 9 (1996): 465-479.


\bibitem{33}
E. Mitidieri, S. Pohozhaev, A priori estimates and blow-up of solutions to nonlinear
partial differential equations, Proc. Steklov Inst. Math., 234 (2001): 1-383.


\bibitem{17}
X. Ros-Oton, J. Serra, The Pohozaev identity for the fractional Laplacian, Arch. Ration.
Mech. Anal. 213(2014): 587-628 .


\bibitem{x1}
X. Ros-Oton, J. Serra, The Dirichlet problem for the fractional laplacian: regularity up to
the boundary, J. Math. Pures Appl. 101(2014): 275-302.


\bibitem{cpam}
L. Silvestre, Regularity of the obstacle problem for a fractional power of the Laplace operator,
Comm. Pure Appl. Math. 60 (2007): 67-112.

\bibitem{r1}
R. Servadei, E. Valdinoci, Variational methods for non-local operators of elliptic type,
Discrete Contin. Dyn. Syst. 33(5)(2013): 2105-2137.


\bibitem{Serrin} 
J. Serrin, A symmetry problem in potential theory. Arch. Rational Mech. Anal. 43 (1971): 304-318.

\bibitem{r2}
L. Silvestre, H$\ddot{\text{o}}$lder estimates for solutions of integro differential equations like the fractional
Laplace, Indiana Univ. Math. J. 55(2006): 1155-1174.










\bibitem{40}
J. Serrin, H. Zou, Existence of positive solutions of the Lane-Emden system, Atti. Sem.
Mat. Fis. Univ. Modena, suppl., 46 (1998): 369-380.










\bibitem{zhuo}
R. Zhuo, W. Chen, X. Cui, Z. Yuan, Symmetry and non-existence of solutions for a nonlinear system
involving the fractional Laplacian, Discrete Contin. Dyn. Syst. 36 (2016): 1125-1141.




\bibitem{aml}
L. Zhang, X. Nie, A direct method of moving planes for the Logarithmic Laplacian, Applied Mathematics Letters, 2021:107141.





















\end{thebibliography}
\end{document}